\documentclass[12pt,reqno]{amsart}
\usepackage{amssymb}
\usepackage{amscd}
\usepackage{hyperref}

\usepackage{amsxtra}

\usepackage{hyperref}

\usepackage{esint}

\usepackage[shortlabels]{enumitem}
\usepackage{cleveref}

\crefrangeformat{enumi}{items #3#1#4--#5#2#6}

\theoremstyle{plain}

\newtheorem{thm}{Theorem}[section]
\newtheorem{lem}[thm]{Lemma}
\newtheorem{cor}[thm]{Corollary}
\newtheorem{prop}[thm]{Proposition}
\newtheorem{mainthm}{Main Theorem}

\theoremstyle{definition}

\newtheorem{defn}[thm]{Definition}

\newtheorem{example}[thm]{Example}

\theoremstyle{remark}

\newtheorem{rem}[thm]{Remark}

\crefname{thm}{theorem}{theorems}
\crefname{lem}{lemma}{lemmas}
\crefname{cor}{corollary}{corollaries}
\crefname{prop}{proposition}{propositions}
\crefname{mainthm}{theorem}{theorems}
\crefname{maincor}{corollary}{corollaries}

\crefname{defn}{definition}{definitions}
\crefname{conj}{conjecture}{conjectures}
\crefname{example}{example}{examples}
\crefname{exercise}{exercise}{exercises}
\crefname{prob}{problem}{problems}
\crefname{quest}{question}{questions}

\crefname{rem}{remark}{remarks}
\crefname{claim}{claim}{claims}
\crefname{axiom}{axiom}{axioms}
\crefname{hyp}{hypothesis}{hypotheses}
\crefname{notation}{notation}{notations}
\crefname{case}{case}{cases}

\numberwithin{equation}{section}

\usepackage{color}
\definecolor{darkgreen}{cmyk}{1,0,1,.2}
\definecolor{m}{rgb}{1,0.1,1}
\newdimen\theight
\def\TeXref#1{%
             \leavevmode\vadjust{\setbox0=\hbox{{\tt
                     \quad\quad  {\small \textrm #1}}}%
             \theight=\ht0
             \advance\theight by \lineskip
             \kern -\theight \vbox to
             \theight{\rightline{\rlap{\box0}}%
             \vss}%
             }}%

\begin{document}
 \title{Transversal Hard Lefschetz theorem  on  transversely symplectic foliations} 

\author[J.A. \'Alvarez L\'opez]{Jes\'us A. \'Alvarez L\'opez}
\address{Departamento de Matem\'aticas\\
         Facultade de Matem\'aticas\\
         University of Santiago de Compostela\\
         15782 Santiago de Compostela\\ Spain}
\email{jesus.alvarez@usc.es}

 \author[S.~D.~Jung]{Seoung Dal Jung}
 \address{Department of Mathematics\\
 Jeju National University\\
 Jeju 63243\\
 Republic of Korea}
 \email{sdjung@jejunu.ac.kr}
 
 \thanks{The first author is partially supported by FEDER/Ministerio de Ciencia, Innovaci\'on y Universidades/AEI/MTM2017-89686-P, and Xunta de Galicia/2015 GPC GI-1574.  The second author was also supported by the National Research Foundation of Korea (NRF) grant funded by the Korea government (MSIP) (NRF-2018R1A2B2002046).}
  

\begin{abstract}  
We study the transversal hard Lefschetz theorem  on a transversely symplectic foliation.  This article extends the results of transversally symplectic flows (H.K.~Pak, ``Transversal harmonic theory for transversally symplectic flows'', J. Aust. Math. Soc. 84 (2008), 233--245) to the general transversely symplectic foliation. 
\end{abstract}
\maketitle

\section{Introduction}
On a compact Riemannian manifold,  the classical Hodge theory states that any cohomology class contains just  one harmonic form because  of  $\mathcal H^r(M) \cong H^r(M)$, where $\mathcal H^r(M)$ and $H^r(M)$ are the harmonic space and de Rham cohomology group, respectively. But on sympleictic manifold, we can not define the harmonic form as in Riemannian case. 

In 1988, J.-L.~Brylinski \cite{BR}  introduced the notion of symplectic harmonic form on a symplectic manifold.
That is, a differential form $\phi$ is  \emph{symplectic harmonic} if $d\phi =\delta\phi=0$, where $\delta$ is a  symplecitc adjoint operator of $d$.   He proved that on a compact K\"ahler  manifold, any de Rham cohomology class has a symplectic harmonic form.  However,  this is not true on an arbitrary symplectic manifold.  In fact,  O.~Mathieu \cite{MA} proved the following hard Lefschetz theorem on a symplectic manifold.

\begin{thm}\label{t: 1}
On a  symplectic manifold $(M,\omega)$ of dimension $2n$, the following  properties are equivalent: 
\begin{enumerate}[{\rm(1)}]

\item\label{i: t-1-(1)} Any cohomology class contains at least one harmonic form.

\item\label{i: t-1-(2)} {\rm(}Hard Lefschtez property.{\rm)} For all $r\leq n$,  the map $L^r:H^{n-r}(M)\to H^{n+r}(M)$ is surjective, where $L[\phi]=[\omega\wedge\phi]$ for any closed form $\phi$.
\end{enumerate}
\end{thm}

\bigskip
The classical Hodge theory on a Riemannian foliation was studied by F.W.~Kamber and Ph.~Tondeur \cite{KT} in 1988.  It states that there is a canonical isomorphism 
\[
\mathcal H_B^r(\mathcal F) \cong H_B^r(\mathcal F)\;,
\]
where $\mathcal H_B^r(\mathcal F)$ and $H_B^r(\mathcal F)$ are the basic harmonic space and basic de Rham cohomology group, respectively.  In a similar way,  a version of \Cref{t: 1} can be considered for transversely symplectic foliations. 

In 2008, H.~Pak \cite{Pa} introduced the transversely symplectic harmonic form  and studied the transverse hard Lefschetz theorem on a transversally symplectic flow, obtaining the following version of \Cref{t: 1} for flows.

\begin{thm}\label{t: 2}
Let $(\mathcal F,\omega)$ be a tense, transversally symplectic flow on a  manifold $M$ of dimension $2n+1$. Then the following  properties are equivalent:
\begin{enumerate}[{\rm(1)}]

\item\label{i: t-2-(1)} Any basic cohomology class for $\mathcal F$ has a transversely symplectic harmonic form.

\item\label{i: t-2-(2)} {\rm(}Transverse hard Lefschetz property.{\rm)} For all $r\leq n$,  the map $L^r:H_B^{n-r}(\mathcal F) \to H_B^{n+r}(\mathcal F)$ is surjective, where $L[\phi]=[\omega\wedge\phi]$ for any basic closed form $\phi$.

\end{enumerate}
\end{thm}

Given a Riemannian metric on $M$, recall that  the foliation $\mathcal F$ is said to be \emph{isoparametric} (respectively, \emph{minimal}) if its mean curvature form is basic (respectively, zero) (see \Cref{s: transv symplectic foln}). If this condition is satisfied for some Riemannian metric, then the foliation is said to be \emph{tense} (respectively, \emph{taut}).  In 2011, L.~Bak and A.~Czarnecki \cite[Corollary~12]{BC}  extended \Cref{t: 2}  to taut  foliations of arbitrary  dimension.   Recently, Y.~Lin \cite{LI} has studied symplectic harmonic theory on transversely symplectic foliations by using the differential operator $d_B$  and $\delta_T$, which is different from $\delta_B$, the symplectic  formal adjoint of $d_B$ (if $\mathcal F$ is minimal, then $\delta_T=\delta_B$).  L.~Bak and A.~Czarnecki  \cite[Theorem~10]{BC} also gave another version of \Cref{t: 2} using $d_B$ and $\delta_T$, without requiring tautness. See also the additional publications \cite{Ib,KP,LI1} about the transversal hard Lefschetz theorem. 

In this paper, we study the transversal hard Lefschetz theorem on an arbitrary transversely symplectic foliation. 
That is, we extend \Cref{t: 2}  to an arbitrary foliation by using the modified symplectic Hodge theory. Let $d_\kappa = d_B-\frac12\kappa\wedge$ be the modified differential operator, where $\kappa$ is the mean curvature form, which $d_B\kappa=0$ on a closed manifold.  Since $d_\kappa^2=0$  on a closed manifold,  we can define the modified basic cohomology group  $H_\kappa^*(\mathcal F)=\ker d_\kappa/\operatorname{im} d_\kappa$, which was introduced by G.~Habib and K.~Richardson \cite{HR}.
Let us point out that  $H_\kappa^*(\mathcal F)$ only depends on the basic cohomology class defined by $\kappa$. Also that, if $\mathcal F$ is isoparametric with any pair of Riemannian metrics on $M$, then the corresponding mean curvature forms induce the same basic cohomology class. Then our main result is the following generalization of \Cref{t: 2}.

\begin{mainthm}[Cf.\ Theorem 4.14]
Let $(\mathcal F,\omega)$ be a transversely symplectic foliation of codimension $2n$ on a  closed manifold $M^{p+2n}$. If $\mathcal F$ is tense, then the following  properties  are equivalent:
\begin{enumerate}[{\rm(1)}]

\item\label{i: mt-(1)} Any modified basic cohomology  class contains at least one modified symplectic harmonic form.

\item\label{i: mt-(2)} For any $r\leq n$,  the homomorphism $L^r: H_\kappa^{n-r}(\mathcal F) \to H_\kappa^{n+r}(\mathcal F)$ is  surjective.

\end{enumerate}
\end{mainthm}

\begin{rem}
\begin{enumerate}[(1)]

\item  The main theorem yields \Cref{t: 2}  when $p=1$ \cite{Pa} and \cite[Corollary~12]{BC}  when $\mathcal F$ is minimal, respectively.   Also, if $\kappa=0$, then $\delta_k=\delta_T$, and therefore our main theorem agrees with \cite[Theorem~10]{BC}  and \cite[Theorem~4.1]{LI}  in this case.

\item  The transversal hard Lefschetz theorem on a transverse  K\"ahler foliation was studied  in \cite{EK} and \cite{JK}. Namely, it was proved that,   on a transverse K\"ahler foliation on a compact Riemannian manifold,  if the mean curvature form satisfies some condition, then the transversal hard Lefschetz property holds. From this fact, we can apply the main theorem to any  transverse K\"ahler foliation on a closed manifold; that is, any basic cohomology class for $\mathcal F$ has a transversely symplectic harmonic representative (cf.\ Corollary 5.3).

\end{enumerate}
\end{rem}

\section{Transversely symplectic foliation}\label{s: transv symplectic foln}

Let $M^{p+2n}$ be a smooth manifold of dimension $m=2n+p$  with a foliation $\mathcal F$.  Let $Q$ be the distribution of rank $2n$ with a nondegenerate closed $2$-form $\omega$ on $Q$ such that  $\ker\omega_x = T\mathcal F_x$ for any $x\in M$, that is, $T\mathcal F =\{\, X\in TM \ |\ i(X)\omega =0 \,\}$  and $\dim T_x\mathcal F=p$. Here $i(X)$ is  the interior product by $X$. The form $\omega$ is said to be \emph{transversely symplectic} with respect to $\mathcal F$. 
 A foliation $\mathcal F$ is said to be \emph{transversely symplectic foliation} if $M$ admits a transversely symplectic form $\omega$ with respect to $\mathcal F$ \cite{LI}. The notation $(\mathcal F,\omega)$ is used in this case.
  Trivially, the form $\omega$ is basic, that is, $\omega\in\Omega_B^2(\mathcal F)$, where 
\[
\Omega_B^*(\mathcal F)=\{\,\phi\in\Omega^*(M)\ |\ i(X)\phi= \theta(X)\phi =0,\ \forall X\in\Gamma T\mathcal F \,\}\;.
\]
Here  $\theta(X)$ is the Lie derivative with respect to $X$.

In particular, if  $\mathcal F$ is the flow generated by a nonsingular vector field $\xi$ such that $i(\xi)\omega=0$,  then $\mathcal F$ is called  the \emph{transversally  symplectic flow} and denoted by $\mathcal F_{\xi}$.
A transverse  K\"{a}hler foliation is a transversely symplectic foliation with  a basic  K\"ahler form as a transverse symplectic form.   For more
examples, see  \cite{Cw,LI}.  

Let $\{v_{1},\cdots ,v_{n},w_{1},\cdots
,w_{n}\}$ be a symplectic frame on $Q$, i.e.,
\[
\omega (v_{a},v_{b})=\omega (w_{a},w_{b})=0\;,\quad \omega (v_{a},w_{b})=\delta_{ab}\;.
\]
Then $\omega$ is locally expressed as 
\[
\omega =\sum_{a=1}^n v_a^* \wedge w_a^*\;,
\]
where $\{v_a^*,w_a^*\}$ is the dual basic frame,  that is, 
\[
v_a^*(X)=\omega(X,w_a)\;,\quad w_a^*(X)=-\omega(X,v_a)\;,
\]
 for any normal vector field $X\in \Gamma Q$.
Any vector field $X\in \Gamma Q$ is expressed as 
\[
X=\sum_{a=1}^n\{\omega(X,w_a)v_a-\omega(X,v_a)w_a\}\;.
\]

Let $\nabla$ be a connection on $Q$. 
  Then the \emph{torsion vector field} $\tau_\nabla$ is given  by
\[
\tau_\nabla =\sum_{a=1}^n T_\nabla(v_a,w_a)\;,
\]
where the torsion tensor $T_\nabla$ is defined by
\[
T_\nabla(X,Y)=\nabla_X\pi(Y)-\nabla_Y\pi(X)-\pi[X,Y]
\]
for any vector fields $X,Y\in\Gamma TM$.
It is easy to prove that the vector field $\tau_\nabla$ is well-defined; that is, it is independent to the choice of symplectic frames on $Q$. 
A  \emph{transversely symplectic connection} $\nabla $  on $Q$ is one which
satisfies $\nabla \omega =0$; that is, for all $X\in \Gamma TM$ and $Y, Z \in \Gamma Q$,
\[
X \omega(Y,Z) = \omega (\nabla_X Y,Z) + \omega(Y,\nabla_XZ)\;.
\]
For the study of an ordinary symplectic manifold,  see  \cite{Bl,Ha}.
 %A connection $\nabla$ is called \rnote{a} \emph{transverse \rnote{transversely} Fedosov connection} if $\nabla$ is a transverse symplectic connection with the torsion free, i.e., $T_\nabla=0$ and $\nabla\omega=0$.
%In this case, the foliation $(\mathcal F,\omega,\nabla)$ \rnote{I deleted ``, where $\nabla$ is the transverse Fedosov connection,'' because it is written ``In this case'' before} is called  \emph{transversely Fedosov foliation}. For \rnote{the study of} an ordinary Fedosov manifold, see \cite{GR} and \cite{KR}.

 Without loss of generality, we assume that $\mathcal F$ is oriented. So, given an auxiliary Riemannian metric on $M$ with $Q=T\mathcal F^\perp$, there is a unique form $\chi_\mathcal F\in\Omega^p(M)$ whose restriction to the leaves is the volume form of the leaves, and such that with $i(Y)\chi_\mathcal F=0$ for any $Y\in\Gamma Q$. The $p$-form $\chi_\mathcal F$ is called the \emph{characteristic form} of $\mathcal F$.  Locally, we can describe $\chi_\mathcal F$ as follows.  Let $\{E_j\}$ ($j=1,\cdots,p$) be a local oriented orthonormal frame of $T\mathcal F$, and let $\{\alpha_j\}$ be the local 1-forms satisfying $\alpha_j(E_k)=\delta_{jk}$ and $\alpha_j(X)=0$ for any $X\in Q$.  Then $\chi_\mathcal F=\alpha_1\wedge\cdots \wedge\alpha_p$; that is,  $\chi_\mathcal F(E_1,\cdots,E_p)=1$. 
Now  the corresponding  mean curvature form $\kappa$ of $\mathcal{F}$  with respect to the auxiliary Riemannian metric is 
\[
\kappa =( -1)^{p+1}i(E_{p})\cdots i(E_{1})d\chi_{\mathcal{F}}\;,
\]%
locally. Since $\varphi _{0}:=d\chi_{\mathcal{F}}+\kappa \wedge \chi_{\mathcal{F}}$ 
satisfies $i(E_p)\cdots i(E_1)\varphi_0 =0$, we get $i\left( X_{1}\right)\cdots i\left( X_{p}\right) \varphi _{0}=0$
for all $X_j \in\Gamma T\mathcal F$.  Now let
\begin{eqnarray*}
F^r\Omega^k:=\{\phi\in \Omega^k(M)\ |\ i(X_1)\cdots
i(X_{k-r+1})\phi=0,\ \forall X_j \in\Gamma T\mathcal F\}\;.
\end{eqnarray*}
Then $\varphi_0 \in F^2 \Omega^{p+1}$, and we have the Rummler's formula   \cite{RU}
\begin{equation}  
d\chi_{\mathcal{F}}=-\kappa \wedge \chi_{\mathcal{F}}+\varphi _{0}\;. \label{Rummler formula}
\end{equation}%

\begin{example}\label{ex: contact manifold}
Let $(M^{2n+1},\alpha)$ be a contact manifold of
dimension $2n+1$, where $\alpha$ is a contact 1-form such that $\alpha\wedge (d\alpha)^n\ne 0$.
Then we have the  Reeb vector field $\xi$ such that $i(\xi)\alpha =1$ and $i(\xi)d\alpha =0$,
which determines a 1-dimensional foliation $\mathcal F_\alpha$, called the \emph{contact flow}. It is trivial that  the contact flow $\mathcal F_\alpha$ on a contact manifold $(M,\alpha)$ is a  transversely symplectic foliation with the transverse symplectic form $\omega:=d\alpha$. Moreover, $\mathcal F_\alpha$ is geodesible, i.e., $\kappa=i(\xi)d\chi_{\mathcal F}=0$, where $\chi_{\mathcal F}=\alpha$. Note that $\varphi_0=d\alpha (\ne 0)$ and  
$TM = T\mathcal F_\alpha  \oplus Q$,  where $Q=\ker \alpha$  (see \cite{Pa}). 
\end{example}

\begin{example}\label{ex: almost cosymplectic manifold}
Let $(M^{2n+1},\eta,\Phi)$ be an almost cosymplectic manifold of dimension $2n+1$, that is, $M$ admits a closed 1-form $\eta$ and a closed 2-form $\Phi$ such that $\eta\wedge \Phi^n$ is a volume form on $M$  \cite{Ib}. Then we have a Reeb vector field $\xi$ satisfying $i(\xi)\eta=1$ and $i(\xi)\Phi =0$.
Let $\mathcal F_\eta$ be  the flow generated by $\xi$, which is called the \emph{cosymplectic flow}. Then the cosymplectic flow $\mathcal F_\eta$ is  a transversely symplectic foliation with a transverse symplectic form $\Phi$.   The characteristic form $\chi_{\mathcal F}$ is given by  $\chi_{\mathcal F}=\eta$. Since $d\eta=0$, 
 $\mathcal F_\eta$ is geodesible, i.e.,  $\kappa=i(\xi)d\chi_{\mathcal F}=0$ and  $\varphi_0 =0$ on $M$. Moreover, since $\eta$ is a closed 1-form, by the Frobenius theorem  $Q=\ker\eta$ also defines  a codimension 1-foliation $\mathcal F_\eta^\perp$ transversal to $\mathcal F_\eta$ (see \cite{Pa}).
\end{example}
 
\begin{prop}[{This is also proved in \cite[Theorem 4.33]{TO}}] \label{p: 2.3} 
Let $(\mathcal F,\omega)$ be a transversely symplectic foliation on a closed oriented manifold. If $\mathcal F$ is taut, then $H^{2r}_B(\mathcal F) \ne 0$ for all $r$.
\end{prop}

\begin{proof}  Assume that $H_B^{2r}(\mathcal F)=0$ for some  $r$. Then  
$\omega^{r}=d\beta$ for some basic $(2r-1)$-form $\beta$. Since $\mathcal F$ is taut, we choose the metric such that $\kappa=0$. So we get
$d\chi_{\mathcal F} = \varphi_0$ and the normal degree of $\beta\wedge\omega^{n-r}\wedge \varphi_0$ is $2n+1$, so it is zero. Hence
\begin{align*}
0&\ne \omega^n \wedge\chi_{\mathcal F}
=d\beta\wedge \omega^{n-r}\wedge\chi_{\mathcal F}
=d(\beta\wedge \omega^{n-r}\wedge\chi_{\mathcal F}) +\beta\wedge\omega^{n-r}\wedge d\chi_{\mathcal F}\\
&=d(\beta\wedge \omega^{n-r}\wedge\chi_{\mathcal F}) +\beta\wedge\omega^{n-r}\wedge \varphi_0
=d(\beta\wedge \omega^{n-r}\wedge\chi_{\mathcal F})\;.
\end{align*}
Thus, by the Stokes' theorem, we get the contradiction
\[
0\ne \int_M \omega^n\wedge\chi_{\mathcal F} =\int_M d(\beta\wedge\omega^{n-r}\wedge\chi_{\mathcal F}) =0\;.\qedhere
\]
\end{proof}

From \Cref{ex: contact manifold,ex: almost cosymplectic manifold}, we get  the following.

\begin{cor} 
Let $(\mathcal F,\omega)$ be a contact or cosymplectic flow on a closed  manifold $M^{2n+1}$. Then $H_B^{2r}(\mathcal F)\ne 0$  for all $r=1,\dots,n$.
\end{cor}

\begin{rem}  
In contrast to an ordinary symplectic manifold,  the condition $H_B^{2r}(\mathcal F)\ne 0$ is not necessary and sufficient  for the existence of a transversely symplectic structure on a foliation. In fact, when $\mathcal F$ is Riemannian and $M$ closed, since $\mathcal F$ is transversely oriented, it is nontaut if and only if $H^{2n}_B(\mathcal F)=0$ \cite{Al}. For example,  consider the hyperbolic torus $T_A^3=T^2\times \mathbb R/\sim$, where $A\in SL(2,\mathbb Z)$ and $(m,t)\sim (A(m),t+1)$ for any $m\in T^2$ and $t\in\mathbb R$.  Then $T_A^3$ has  a transversely symplectic foliation $\mathcal F$ of codimension 2 such that  $\mathcal F$ is nontaut and $H_B^2(\mathcal F)=0$ \cite[Example~9.1]{JK}.
\end{rem}

\section{Transversely symplectic harmonic forms}
Let $(\mathcal F,\omega)$ be a transversely symplectic foliation.
Let $\flat: Q\to Q^*$ be defined by $\flat(X) = i(X)\omega$.
Since $\omega$  plays a role of  a symplectic structure on the distribution $Q$, the map $\flat$ is an isomorphism.
It is trivial that $\flat(v_a) = w_a^*$ and $\flat (w_a) =- v_a^*$.
Let $\sharp=\flat^{-1}$.  For any $\phi\in\Gamma Q^*$,
\begin{align}\label{3-1}
\phi = i(\phi^\sharp)\omega\;,
\end{align}
where  $\phi^\sharp:=\sharp(\phi)$.  From~\eqref{3-1}, we  get
\begin{align}\label{3-2}
\omega(\phi^\sharp,\psi^\sharp) = i(\psi^\sharp)\phi
\end{align}
for all $\phi,\psi\in\Gamma Q^*$. Moreover  $\phi$ and $\phi^\sharp$ can be locally expressed as
\[
\phi = \sum_{a=1}^n \{ \phi(w_a)\flat(v_a)-\phi(v_a)\flat(w_a)\}\;,\quad \phi^\sharp =\sum_{a=1}^n\{\phi(w_a)v_a -\phi(v_a)w_a\}\;.
\]
The map $\flat$ can be extended to an isomorphism $\flat:\Gamma\Lambda^r Q\to\Gamma\Lambda^r Q^*$, defined by 
\begin{equation}\label{3-4}
\flat(X_1\wedge \cdots\wedge X_r) := \flat(X_1)\wedge\cdots\wedge \flat(X_r)\;,
\end{equation}
where $X_i \in\Gamma Q$ ($i=1,\dots,r$). Similarly, $\sharp$ is extended to an isomorphism $\Gamma\Lambda^r Q^*\to\Gamma\Lambda^r Q$.  Now let 
\[
\omega(\phi,\psi) =\det\big(\omega(\phi_i,\psi_j)\big)_{i,j=1,\dots,r}
\]
 for all   $r$-forms $\phi=\phi_1\wedge\cdots\wedge\phi_r$ and $\psi=\psi_1\wedge\cdots\wedge\psi_r$ with  $\phi_j, \psi_j \in\Gamma Q^*$  $(j=1,\dots,r)$, 
where $\omega(\phi_i,\psi_j) = \omega(\phi_i^\sharp,\psi_j^\sharp)$.  Then 
\begin{align} \label{3-6}
\omega(\phi,\psi) = (-1)^r \omega(\psi,\phi)\;,\quad  i(\phi^\sharp)\psi = \omega(\psi,\phi)\;,
\end{align}
where $i(\phi^\sharp) = i(\phi_r^\sharp)\cdots i(\phi_1^\sharp)$.
The space of  basic vector fields is
\[
\mathfrak X_B (\mathcal F)=\{\,X \in \Gamma Q\ |\ [X,V] \in \Gamma T\mathcal F, \ \forall V\in  \Gamma T\mathcal F\,\}\;.
\]

\begin{lem} \label{l: 3-1}
Let $(\mathcal F,\omega)$ be transversely symplectic foliation on a smooth manifold $M$. Then $\flat$ defines an isomorphism $\mathfrak X_B(\mathcal F) \cong \Omega_B^1(\mathcal F)$.
\end{lem}
\begin{proof}  Let $X\in \mathfrak X_B(\mathcal F)$. For any vector field $E\in\Gamma T\mathcal F$, 
\begin{equation}\label{3-7}
i(E)\flat(X) = i(E) i(X)\omega =-i(X)i(E)\omega=0\;.
\end{equation}
Note that $\theta(Y) = di(Y) + i(Y)d$ and $[\theta(Y),i(Z)] = i([Y,Z])$ for any $Y$ and $Z$. Hence
\begin{equation}\label{3-8}
i(E) d(\flat(X)) = i(E)d i(X)\omega = i(E)\theta(X)\omega = \theta(X) i(E)\omega - i([X,E])\omega =0\;.
\end{equation}
The last equality in~\eqref{3-8} holds because $X\in\mathfrak X_B(\mathcal F)$. By ~\eqref{3-7} and~\eqref{3-8}, $\flat(X)\in\Omega_B^1(\mathcal F)$.

Conversely, let $\phi\in\Omega_B^1 (\mathcal F)$.  Then there exists $X\in Q$ such that $\flat(X)=\phi$, i.e., $\phi = i(X)\omega$.  For $E\in\Gamma T\mathcal F$, since $i(E)\omega=0$ and $d\phi$ is basic, we have
\begin{align*}
i([X,E]) \omega &=\theta(X) i(E)\omega -i(E)\theta(X)\omega
= -i(E)\theta(X)\omega\\
&=-i(E) di(X)\omega
=-i(E) d\phi 
=0\;;
\end{align*}
that is, $[X,E]\in\Gamma T\mathcal F$. Hence $X\in \mathfrak X_B(\mathcal F)$.
\end{proof}
From \Cref{l: 3-1} and~\eqref{3-4}, $\flat$ can be naturally extended to an isomorphism $\flat : \mathfrak X_B^r(\mathcal F) \to \Omega_B^r (\mathcal F)$,
where $\mathfrak X_B^r(\mathcal F)$ is the space of all  foliated skew symmetric $r$-vector fields.

Now, let $T^*:\Gamma\Lambda Q^* \to \Gamma\Lambda Q^*$ denote the \emph{symplectic adjoint operator} of any  operator $T:\Gamma\Lambda Q^*\to \Gamma\Lambda Q^*$, which is given by
\[
\omega(T\phi,\psi) = \omega(\phi,T^*\psi)
\]
for any $\phi,\psi\in \Lambda Q^*$.  The \emph{transversal symplectic Hodge star operator} $\bar\star :\Lambda^r Q^* \to \Lambda^{2n-r}Q^*$   is defined by the formula
\[
\phi\wedge\bar\star\,\psi = \omega(\phi,\psi)\nu
\]
for any $\phi,\psi \in \Gamma\Lambda^r Q^*$,
where $\nu=\omega ^{n}/n!$ is the transversal volume form of $\mathcal{F}$.

\begin{lem}\label{l: 3-2}
For any $X\in\Gamma Q$, $\alpha\in\Gamma Q^*$ and $\phi\in\Gamma \Lambda^r Q^*$,
\[
(1)\ \bar\star\, \phi = i(\phi^\sharp)\nu\;,  \quad  (2)\  \bar\star\,\epsilon(X^{\flat}) \phi = (-1)^r i(X) \,\bar\star\, \phi\;, \quad
   (3)\   \epsilon(\alpha)^*= - i(\alpha^\sharp)\;,\quad  (4)\ (\bar\star)^2= \operatorname{id}\;,
\]
where $\epsilon(\alpha)=\alpha\wedge$ is the exterior product. 
\end{lem}

\begin{proof}
(1) This equality is proved by induction on $r$.   For $\phi,\psi\in\Gamma Q^*$, we have
\[
\phi\wedge i(\psi^\sharp)\nu= - i(\psi^\sharp)(\phi\wedge \nu) + (i(\psi^\sharp)\phi)\nu= (i(\psi^\sharp)\phi)\nu=\omega(\phi,\psi)\nu= \phi\wedge\bar\star\,\psi\;.
\]
Assume that  it holds for $r-1$.  Let $\psi = \eta\wedge\beta$ for $\eta\in\Gamma\Lambda^{r-1}Q^*$ and $\beta\in \Gamma Q^*$. Then
\begin{align*}
\phi\wedge i(\psi^\sharp)\nu
&= (-1)^r i(\beta^\sharp)\big(\phi\wedge i(\eta^\sharp)\nu\big) +(-1)^{r+1} i(\beta^\sharp)\phi \wedge i(\eta^\sharp)\nu\\
&=(-1)^{r+1}i(\beta^\sharp)\phi \wedge i(\eta^\sharp)\nu
=(-1)^{r+1} \big(i(\eta^\sharp)i(\beta^\sharp)\phi\big)\nu \\
&= (-1)^{r+1} \big(i(\beta^\sharp\wedge\eta^\sharp)\phi\big)\nu
=\big(i(\eta^\sharp\wedge\beta^\sharp)\phi\big)\nu\\
&=\big(i(\psi^\sharp)\phi\big)\nu 
=\omega(\phi,\psi)\nu
=\phi\wedge\bar\star\,\psi\;.
\end{align*}
Here, $\phi\wedge i(\eta^\sharp)\nu$ is zero because it is of degree $q+1$.  The third equality holds because $i(\beta^\sharp)\phi $ is of degree $r-1$.

 (2)  This equality follows from (1).  
 
  (3) For any $\phi\in\Gamma\Lambda^{r-1}Q^*$ and $\psi\in\Gamma\Lambda^r  Q^*$,  we have 
\begin{align*}
\omega(\epsilon(\alpha)\phi,\psi)\nu&=\epsilon(\alpha)\phi\wedge\bar\star\,\psi
=(-1)^{r-1}\phi\wedge\epsilon(\alpha)\,\bar\star\,\psi
=(-1)^{r-1}\phi\wedge\bar\star\,(\bar\star\,\epsilon(\alpha)\,\bar\star\,\psi)\\
&=\omega(\phi,(-1)^{r-1}\,\bar\star\,\epsilon(\alpha)\,\bar\star\,\psi)\nu
=-\omega(\phi, i(\alpha^\sharp)\psi)\nu\;,
\end{align*}
which proves (3).  

 (4) Note that,  for any $\phi\in\Gamma \Lambda^r Q^*$ and $\psi\in\Gamma\Lambda^{q-r}Q^*$,
\[
\phi\wedge\psi = \omega(\phi\wedge\psi,\nu)\nu\;.
\]
Since $q=2n$, we obtain
\[
\phi\wedge (\bar\star)^2 \psi = \omega(\phi,\bar\star\,\psi)\nu
=\omega(\phi, i(\psi^\sharp)\nu)\nu=(-1)^r \omega(\psi\wedge\phi,\nu)\nu
=(-1)^r \psi\wedge\phi = \phi\wedge\psi\;.
\]
Hence $(\bar\star)^2= \operatorname{id}$. 
\end{proof}

\begin{prop} 
Let $(\mathcal F,\omega)$ be  transversely symplectic foliation on a smooth manifold $M$.  Then   $\bar\star$ induces a homomorphism $\bar\star:\Omega_B^r(\mathcal F)\to \Omega_B^{2n-r}(\mathcal F)$.
\end{prop}
\begin{proof} Let $\phi\in\Omega_B^r(\mathcal F)$.  Since $\nu$ is basic,  by  \Cref{l: 3-2}~(1),  it is trivial that   $i(X) \,\bar\star\,\phi =0$ for any $X\in \Gamma T\mathcal F$.  On the other hand,   for any $X\in \Gamma T\mathcal F$ and $K\in \mathfrak X_B^r(\mathcal F)$,  $[\theta(X), i(K)] = i(\theta(X)K) $ and $\theta(X)K \in \mathfrak X_B^r (\mathcal F)$.  Since $\nu$ is basic, for any $X\in\Gamma T\mathcal F$,
\[
\theta(X)\,\bar\star\,\phi = \theta(X)i(\phi^\sharp)\nu 
=i(\phi^\sharp)\theta(X)\nu + i(\theta(X) \phi^\sharp)\nu
=0\;.
\]
Hence $\bar\star\, \phi$ is basic.  
\end{proof}

Let $d_B:=d|_{\Omega_B^r(\mathcal F)}$, and  define $\delta_B$ on $\Omega_B^r(\mathcal F)$ by 
\[
\delta_B \phi = (-1)^r \,\bar\star\, (d_B -\kappa\wedge)\,\bar\star\, \phi\;.
\]
If $\mathcal F$ is isoparametric, i.e., $\kappa\in\Omega_B^1(\mathcal F)$, then $\delta_B$ preserves the basic forms.

\begin{prop}\label{p: 3-4}
Let $(\mathcal F,\omega)$ be a transversely symplectic foliation on a closed manifold $M$. If $\mathcal F$ is isoparametric,  then 
\[
\int_M\omega(d_B \phi,\psi)\mu_M = \int_M\omega(\phi,\delta_B\psi)\mu_M
\]
 for all basic forms $\phi\in\Omega_B^{r-1}(\mathcal F)$ and $\psi\in\Omega_B^{r}(\mathcal F)$, using the volume form $\mu_M=\nu\wedge\chi_\mathcal F$ on $M$.
\end{prop}
\begin{proof} Let  $\phi \in \Omega_B^{r-1}(\mathcal F)$ and $\psi\in\Omega_B^r(\mathcal F)$. Since the normal degree of $\phi\wedge \bar\star\,\psi\wedge\varphi_0$ is $q+1$, it is zero. Hence, by the Stokes' theorem and the Rummler's formula~\eqref{Rummler formula}, we have
 \begin{align*}
\int_M\omega(d_B\phi,\psi)\mu_M&=\int_M d\phi\wedge \bar\star\,\psi\wedge\chi_{\mathcal F}\\
&= \int_M d(\phi\wedge \bar\star\,\psi\wedge \chi_{\mathcal F})  -(-1)^{r-1}\int_M \phi\wedge d\,\bar\star\,\psi\wedge\chi_{\mathcal F} +\int_M \phi\wedge\bar\star\,\psi\wedge d\chi_{\mathcal F}\\
&=(-1)^{r}\int_M \phi\wedge d\,\bar\star\,\psi\wedge\chi_{\mathcal F}- \int_M \phi\wedge\bar\star\,\psi \wedge\kappa\wedge\chi_{\mathcal F} +\int_M \phi\wedge\bar\star\,\psi \wedge\varphi_0\\
&=(-1)^{r}\int_M \phi\wedge (\bar\star)^2 d\,\bar\star\, \psi\wedge\chi_{\mathcal F} +(-1)^{r+1}\int_M \phi\wedge (\bar\star)^2(\kappa\wedge\bar\star\,\psi)\wedge\chi_{\mathcal F}\\
&=(-1)^{r}\int_M \phi\wedge \bar\star\,(\bar\star\, d\bar\star -\bar\star\,\kappa\wedge \bar\star)\psi \wedge \chi_{\mathcal F}\\
&=(-1)^{r}\int_M \omega(\phi, \bar\star\, (d-\kappa\wedge)\,\bar\star\,\psi)\nu\wedge\chi_{\mathcal F}
=\int_M\omega(\phi,\delta_B\psi)\mu_M\;.\qedhere
\end{align*} 
\end{proof}

\begin{rem} 
By ~\eqref{3-6},  it is trivial that
\[
\int_M\omega (\delta_B\phi,\psi)\mu_M =-\int_M\omega(\phi,d_B \psi)\mu_M\;.
\]
\end{rem} 
Now, let $\delta_T= (-1)^{r} \,\bar\star\, d_B\,\bar\star$ on $\Omega_B^r(\mathcal  F)$.
Since $(\bar\star)^2 = \operatorname{id}$, we have the following.

\begin{lem}  
On $\Omega_B^r (\mathcal F)$,  we have
\[
 d_B \bar\star = (-1)^r \,\bar\star\,\delta_T\;,\quad  \bar\star\, d_B = (-1)^r \delta_T\,\bar\star\;.
\]
\end{lem}

\begin{lem}\label{l: 3-7}
If $\mathcal F$  is isoparametric,  i.e., $\kappa\in\Omega_B^1(\mathcal F)$, 
then, on $\Omega_B^*(\mathcal F)$,
\[
(1) \ \delta_B= \delta_T-i(\kappa^\sharp)\;, \quad
(2) \  \delta_T \epsilon(\kappa) =- \bar\star\, d_B i(\kappa^\sharp) \bar\star\;,\quad (3)\  \epsilon(\kappa)\delta_T = -\bar\star\, i(\kappa^\sharp) d_B\bar\star\;.
\]
In addition, if $M$ is closed, then 
\begin{align}\label{3-15}
\ \delta_T i(\kappa^\sharp) + i(\kappa^\sharp)\delta_T =0\;.
\end{align} 
\end{lem}
\begin{proof}   
By  \Cref{l: 3-2}~(2), the equality~(1) is trivial. 
 
Let us prove (2) and (3). By \Cref{l: 3-2}~(2), 
\begin{align*}
\delta_T\epsilon(\kappa)\phi&= (-1)^{r+1} \,\bar\star\, d_B \,\bar\star\, (\epsilon(\kappa)\phi)
=-\bar\star\, d_B i(\kappa^\sharp) \,\bar\star\, \phi\;,\\
\epsilon(\kappa)\delta_T\phi &=(-1)^{r}\epsilon(\kappa)\,\bar\star\, d_B\,\bar\star\,\phi
=(-1)^{r} \,\bar\star\,(\bar\star\,\epsilon(\kappa)\,\bar\star\,(d_B\,\bar\star\,\phi))
=- \bar\star\, i(\kappa^\sharp)d_B\,\bar\star\,\phi\;.
\end{align*}
On the other hand, if $M$ is closed, then $\delta_B^2=0$  by \Cref{p: 3-4}.   Hence 
\[
\delta_T i(\kappa^\sharp) + i(\kappa^\sharp)\delta_T=-\delta_B^2=0\;,
\]
which proves~\eqref{3-15}.
\end{proof}

By \Cref{l: 3-7}, we have the following.

\begin{cor}   
If $\mathcal F$  is isoparametric, then
\[
\delta_T \epsilon(\kappa) +\epsilon(\kappa)\delta_T = -\bar\star\, \theta(\kappa^\sharp)\,\bar\star\;.
\]
\end{cor}

For  later use, we prove the following about the mean curvature form.

\begin{prop}\label{p: 3-9}
If $\mathcal F$ is isoparametric on a closed manifold $M$, then
 $d\kappa=0$.
\end{prop}
\begin{proof}   
Let $\tilde d = d-\kappa\wedge$. Then $0=\delta_B^2 = -\bar\star\, (\tilde d)^2 \,\bar\star$,
and therefore $(\tilde d )^2=0$ and $\tilde d (1) = -\kappa$. Hence 
\[
d\kappa = (\tilde d +\kappa\wedge)\kappa =\tilde d\kappa =- \tilde d (\tilde d 1) =0\;.\qedhere
\]
\end{proof}

\section{Transversal hard Lefschetz theorem}
Let $(\mathcal F,\omega)$ be a transversely symplectic foliation on a smooth manifold $M$.
Now, we define the operator $L:\Gamma\Lambda^r Q^*\to\Gamma \Lambda^{r+2}Q^*$  by $L\phi = \omega\wedge\phi$ for any  form $\phi\in\Gamma\Lambda^rQ^*$. Let $\Lambda=L^*:\Gamma\Lambda^r Q^*\to\Gamma \Lambda^{r-2}Q^*$ be the symplectic adjoint operator of $L$; that is, 
\[
\omega(L\phi,\psi) = \omega(\phi,\Lambda\psi)\;.
\] 

\begin{lem}[Cf.\ \cite{Pa}]  \label{l: 4-1}
Let $(\mathcal F,\omega)$ be a transversely symplectic foliation. Then, on $\Gamma \Lambda^rQ^*$, 
\[ 
\Lambda= \bar\star\, L\,\bar\star=i(\omega^\sharp)\;,\quad  L\,\bar\star = \bar\star\, \Lambda\;.
\]
\end{lem}

\begin{proof} 
Let $\phi\in\Lambda^r Q^*$ and $\psi\in\Lambda^{r+2}Q^*$.  Then
\[
\omega(L\phi,\psi)\nu=L\phi\wedge\bar\star\,\psi
=\phi\wedge L\,\bar\star\,\psi
=\phi\wedge \bar\star\,(\bar\star\, L\,\bar\star\,\psi)
= \omega(\phi, \bar\star\, L\,\bar\star\,\psi)\nu\;,
\]
which proves  $\Lambda = \bar\star\, L\,\bar\star$.
By \Cref{l: 3-2},  we have
\[
\bar\star\, L \,\bar\star\,\phi = \bar\star\,(\omega\wedge \bar\star\,\phi)
= i((\bar\star\,\phi)^\sharp)i(\omega^\sharp)\nu
=i(\omega^\sharp)i((\bar\star\,\phi)^\sharp)\nu
= i(\omega^\sharp)(\bar\star)^2\phi
= i(\omega^\sharp)\phi\;.
\]
Finally,  $\bar\star\,\Lambda = (\bar\star)^2 L\,\bar\star = L\,\bar\star$.  
\end{proof}

\begin{rem}  
Since $\bar\star$ preserves the basic forms,  $L$ and $\Lambda$ preserve the basic forms. 
\end{rem}

\begin{lem} \label{l: 4-3}
Let $(\mathcal F,\omega)$ be a transversely symplectic foliation on $M$. Then $[d_B,\Lambda]=\delta_T$ on $\Omega_B^*(\mathcal F)$
\end{lem}

\begin{proof}  
The proof is similar to the proof of \cite[Theorem~2.2.1]{BR}.
\end{proof}

We introduce the operator $A:\Omega^*_B(\mathcal F)\to \Omega_B^*(\mathcal F)$ defined by
\begin{equation}\label{4-3}
A=\sum_{r=0}^{2n} (n-r)\pi_r\;,
\end{equation}
where $\pi_r:\Omega_B^*(\mathcal F)\to \Omega_B^r (\mathcal F)$ is the natural projection.
Then we have the following  \cite{BR,Ya}.

\begin{lem} \label{l: 4-4}
Let $(\mathcal F,\omega)$ be a transversely symplectic foliation on $M$.  Then, on $\Omega_B(\mathcal F)$:
\begin{enumerate}[{\rm(1)}]

\item $ [\Lambda,L] = A,\  [A,L]=-2L,\  [A,\Lambda]=2\Lambda$;

\item $[L,d_B]=[\Lambda,\delta_T]=0,$   $[L,\epsilon(\kappa)]=[\Lambda,i(\kappa^\sharp)]=0$;

\item $[A,\delta_B]=\delta_B, \  [L,\delta_T]=[A,d_B]=-d_B$; and

\item  $[L, i(X)] = -\epsilon(X^\flat)$ and $[\Lambda,\epsilon(X^\flat)]=-i(X)$  for any $X\in\mathfrak X_B(\mathcal F)$.

\end{enumerate}
\end{lem}

\begin{proof} (1) Using \Cref{l: 4-1},  the proof is easy.  

(2) Trivially, $[L,d_B]=[L,\epsilon(\kappa)]=0$.  By the definition of $\delta_T$ and \Cref{l: 4-1,l: 3-2}, 
\begin{gather*}
\Lambda \delta_T= (-1)^{r} \,\bar\star\, Ld_B \,\bar\star= (-1)^{r}\,\bar\star\, d_B L \,\bar\star=(-1)^{r} \,\bar\star\, d_B \,\bar\star\, \Lambda=\delta_T\Lambda\;,\\
\Lambda i(\kappa^\sharp)= (-1)^{r}\,\bar\star\, L\epsilon(\kappa)\,\bar\star
= (-1)^{r} \,\bar\star\, \epsilon(\kappa)L\,\bar\star
=(-1)^{r} \,\bar\star\,\epsilon(\kappa)\,\bar\star\,\Lambda
= i(\kappa^\sharp)\Lambda\;,
\end{gather*}
on $\Omega_B^r(\mathcal F)$. So $[\Lambda,\delta_T]=[\Lambda,i(\kappa^\sharp)]=0$.   

(3) The proofs of $[A,\delta_B]=\delta_B$  and $[A,d_B]=-d_B$  are trivial by~\eqref{4-3}.  By \Cref{l: 4-3},
\[
[L,\delta_T]= [\Lambda,L]d_B - d_B[\Lambda,L] =[A, d_B]\;.
\]

(4)  For any $X\in\Gamma Q$ and $\phi\in\Omega_B^r(\mathcal F)$, 
\[
i(X)L\phi = i(X)(\omega\wedge\phi) = i(X)\omega\wedge\phi + Li(X)\phi\;,
\]
which proves the first equality.  The second one follows because,  by \Cref{l: 3-2,l: 4-1},  for any $\phi\in \Omega_B^r(\mathcal F)$,
\begin{align*}
[\Lambda,\epsilon(X^\flat)]\phi&= \Lambda \epsilon(X^\flat)\phi-\epsilon(X^\flat)\Lambda\phi\\
&=(-1)^{r}\,\bar\star\, L (\bar\star)^2 i(X) \,\bar\star\,\phi - (-1)^r \,\bar\star\, i(X) (\bar\star)^2 L\,\bar\star\, \phi\\
&=(-1)^{r} \,\bar\star\,[L,i(X)]\,\bar\star\,\phi
=(-1)^{r+1} \,\bar\star\,\epsilon(X^\flat)\,\bar\star\,\phi
= - i(X)\phi\;.\qedhere
\end{align*}  
\end{proof}

From \Cref{l: 4-3}, it is well known that it follows that $d_B\delta_T +\delta_T d_B =0$. Hence
\[
\Delta_B:=d_B\delta_B + \delta_B d_B =-\theta(\kappa^\sharp)\;.
\]
If $\mathcal F$ is minimal, then $d_B\delta_B+\delta_B d_B =0$; that is, $\ker\Delta_B=\Omega_B(\mathcal F)$.  Thus we define the following. 

\begin{defn}
A basic form $\phi$ is said to be a \emph{transversely} (resp., \emph{normally}) \emph{symplectic harmonic form} if $d_B\phi=\delta_B\phi=0$ (resp., $d_B\phi=\delta_T\phi=0$).
\end{defn}

\begin{rem}
\Cref{l: 4-4} implies that  $\{A,L,\Lambda\}$ spans the Lie algebra 
$\operatorname{sl}(2)$. Hence the space $\Omega_B^*(\mathcal F)$ is a $\operatorname{sl}(2)$-module  on which $A$ acts diagonally with only finitely many different eigenvalues.  Hence we have the duality on transversely symplectic harmonic forms \cite{Pa}. 
\end{rem}

Let  $\mathcal H_{SB}^r (\mathcal F)$ (resp., $\mathcal H_{ST}^r(\mathcal F)$)  be the space of all transversely (resp., normally) symplectic harmonic forms on $M$; that is, 
\begin{align*}
\mathcal H_{SB}^r (\mathcal F) &= \{\,\phi\in\Omega_B^r(\mathcal F)\ |\ d_B\phi = \delta_B\phi=0\,\}\;,\\
\mathcal H_{ST}^r(\mathcal F) &=\{\,\phi \in\Omega_B^r (\mathcal F)\ |\  d_B\phi = \delta_T\phi =0\,\}\;.
\end{align*}
If the foliation is minimal, then $\mathcal H_{SB}^*(\mathcal F) = \mathcal H_{ST}^*(\mathcal F)$.

\begin{prop}\label{p: 4-7}
Let $(\mathcal F,\omega)$ be a transversely symplectic foliation.  Then $L^r :\mathcal H_{ST}^{n-r}(\mathcal F) \to \mathcal H_{ST}^{n+r}(\mathcal F)$ is an isomorphism. 
\end{prop}

\begin{proof} 
Since $[L,d_B]=0$ and $[L,\delta_T]=-d_B$, the proof  is easy.  
\end{proof}

\begin{rem} 
By \Cref{p: 4-7}, Y.~Lin \cite{LI} studied the existence of normally symplectic harmonic representatives in a basic cohomology class on a transversely symplectic foliation. 
But the operator $\delta_T$ is not a symplectic adjoint operator of $d_B$ when the foliation is not minimal. So it is natural to consider the symplectic formal adjoint $\delta_B$ of $d_B$ instead of $\delta_T$.  But we may have $[L,\delta_B]\ne d_B$ (in fact,  $[L,\delta_B]=-d_B+\epsilon(\kappa)$). Hence, on an isoparametric folitaion, $L^r $ may not preserve the transversely symplectic harmonic forms generally.  But we can overcome this problem by modifying the operators $d_B$ and $\delta_B$ as follows. 
\end{rem}

Now, we consider the modified operators 
\[
d_\kappa = d_B - \frac12 \epsilon(\kappa)\;,\quad \delta_\kappa = \delta_B +\frac12 i(\kappa^\sharp)\;.
\]
It is trivial that, if $\mathcal F$ is isoparametric, then $d_\kappa$ and $\delta_\kappa$ preserve the basic forms.

\begin{prop}\label{p: 4-9}
Let $(\mathcal F,\omega)$ be a transversely symplectic  foliation on a closed manifold $M$. If $\mathcal   F$ is isoparametric, then  $\delta_\kappa$ is the symplectic adjoint operator of $d_\kappa$, i.e.,  for all $\phi\in\Omega_B^r(\mathcal F)$ and $\psi\in\Omega_B^{r+1}(\mathcal F)$,
\[
\int_M\omega(d_\kappa\phi,\psi)\mu_M =\int_M\omega (\phi,\delta_\kappa \psi)\mu_M\;.
\]
\end{prop}
\begin{proof} From \Cref{l: 3-2} and \Cref{p: 3-4}, the proof follows easily.  
\end{proof}

\begin{lem}  \label{l: 4-10}
Let  $(\mathcal F,\omega)$ be a transversely symplectic foliation.  If $\mathcal F$ is isoparametric, then 
\begin{align}\label{4-7}
[A,d_\kappa]= [L,\delta_\kappa]=-d_\kappa\;,\quad  [A,\delta_\kappa]=[d_\kappa,\Lambda]=\delta_\kappa\;,\quad
[L,d_\kappa]=[\Lambda,\delta_\kappa]=0\;.
\end{align}
In particular, if $M$ is closed, then $d_\kappa^2 = \delta_\kappa^2 =0$.
\end{lem}

\begin{proof}  
The proof of~\eqref{4-7} follows from~\Cref{l: 4-4}.  From \Cref{p: 3-9}, we get $d_\kappa^2=0$, and therefore $\delta_k^2=0$ by \Cref{p: 4-9}.  
\end{proof}

Note that  $\Delta_k:=d_\kappa\delta_\kappa + \delta_\kappa d_\kappa =0$ because of $\delta_k=[d_k,\Lambda]$ by \Cref{l: 4-10}.   Then  $\ker\Delta_\kappa=\Omega_B(\mathcal F)$.  Hence we define the following.

\begin{defn}
A basic form $\phi$ is said to be a \emph{modified symplectic harmonic form} if $d_\kappa\phi=0$ and $\delta_\kappa\phi=0$.  And the \emph{modified symplectic harmonic space} is defined by
\[
\mathcal H_{SK}^r(\mathcal F) =\{\,\phi\in\Omega_B^r(\mathcal F)\ | \ d_\kappa\phi =\delta_\kappa \phi=0\,\}\;.
\]
\end{defn}

\begin{prop}  \label{p: 4-12}
Let $(\mathcal F,\omega)$ be a transversely symplectic foliation on $M^{p+2n}$. If $\mathcal F$ is isoparametric, then $L^r : \mathcal H_{SK}^{n-r}(\mathcal F)\to \mathcal H_{SK}^{n+r}(\mathcal F)$ is an isomorphism.
\end{prop}

\begin{proof} 
By \Cref{l: 4-10}, the proof is easy. 
\end{proof}

On a closed manifold,  $d_\kappa^2=0$ by \Cref{l: 4-10}.  So  the cohomology group $H_\kappa^r(\mathcal F)=\ker d_\kappa/\operatorname{im}d_\kappa$ (called as {\it modified basic cohomology group}) is defined on a closed manifold.  See \cite{HR} for many properties of this  modified basic cohomology. 
Let
\[
\widetilde H_\kappa^r (\mathcal F) =\{\, [\phi]\in H_\kappa^r(\mathcal F) \ |\ d_k\phi = \delta_k\phi =0\,\}\;.
\]
A basic form $\phi$ is said to be \emph{primitive} if $\Lambda\phi=0$. Then  a basic form $\phi \in\Omega_B^r(\mathcal  F)$ is primitive if and only if 
$L^{n-r+1}\phi=0$.
Let
\begin{align*}
P\Omega_B^r(\mathcal F) &=\{\,\phi\in\Omega_B^r(\mathcal F)\  | \ L^{n-r+1}\phi=0\,\}\;,\\
PH_\kappa^r(\mathcal F) &=\{[\phi]\in H_\kappa^r(\mathcal F)\  | \ L^{n-r+1}[\phi]=0\}\;.
\end{align*}

\begin{thm}  \label{t: 4-13}
Let $(\mathcal F,\omega)$ be a transversely symplectic foliation of codimension $2n$ on a  closed manifold $M^{p+2n}$. If $\mathcal F$ is tense, then the following properties are equivalent:
\begin{enumerate}[{\rm(1)}]

\item Any basic cohomology class contains at least one modified symplectic harmonic form, that is, $\widetilde H_\kappa^* (\mathcal F) = H_\kappa^*(\mathcal F)$.

\item For any $r\leq n$,  the homomorphism $L^r: H_\kappa^{n-r}(\mathcal F) \to H_\kappa^{n+r}(\mathcal F)$ is surjective.

\end{enumerate} 
\end{thm}

\begin{proof}
The proof is similar to ones in \cite{Pa} and \cite{Ya}.   In fact,  assume  $\widetilde H_\kappa^* (\mathcal F) = H_\kappa^*(\mathcal F)$. So the canonical map $\mathcal H_{SK}^r(\mathcal F) \to H_\kappa^r(\mathcal F)$ is surjective. Hence, by \Cref{p: 4-12},  it is trivial that $L^r : H_\kappa^{n-r} (\mathcal F) \to H_\kappa^{n+r}(\mathcal F)$ is surjective.   

Conversely, assume that for any $r\leq n$,  the map $L^r: H_\kappa^{n-r}(\mathcal F)\to H_\kappa^{n+r}(\mathcal F)$ is surjective.  By induction on $r$, we prove  that $\widetilde H_\kappa^r(\mathcal F) = H_\kappa^r(\mathcal F)$. First, it is trivial that  $\widetilde H_\kappa^0(\mathcal F) = H_\kappa^0 (\mathcal F)$. 

For $r=1$, let $[\phi]\in H_\kappa^1(\mathcal F)$. We have $d_k\phi=0$  and $\delta_k\phi = [d_k,\Lambda]\phi =0$ because $\Lambda\phi=0$.  Hence $\widetilde H_\kappa^1 (\mathcal F) = H_\kappa^1(\mathcal F)$. 

Now, assume that this property holds for $s<n-r$, and let us show that 
\begin{equation} \label{4-13}
\widetilde H_\kappa^{n-r}(\mathcal F) = H_\kappa^{n-r}(\mathcal F)\;.
\end{equation}
Trivially, $\widetilde H_\kappa^{n-r}(\mathcal F)\subset H_\kappa^{n-r}(\mathcal F)$.  Let $[\phi]\in H_\kappa^{n-r}(\mathcal F)$.  Since $L^{r+2}:H_\kappa^{n-r-2}(\mathcal F)\to H_\kappa^{n+r+2}(\mathcal F)$ is surjective by assumption,  there exists $[\psi]\in H_\kappa^{n-r-2}(\mathcal F)$ such that $L^{r+1}[\phi] = L^{r+2}[\psi]$; that is, $L^{r+1}([\phi]-L[\psi])=0$.   Hence $[\phi] -L[\psi] \in PH_\kappa^{n-r}(\mathcal F)$.  So
\[
[\phi] = ([\phi] - L[\psi]) + L[\psi]\;,
\]
which means
\[
H_\kappa^{n-r}(\mathcal F) =  PH_\kappa^{n-r}(\mathcal F) + \operatorname{im} L\;.
\]
By induction and \Cref{l: 4-10}, it is well known that any class of $\operatorname{im} L$ contains a modified symplectic harmonic representative; that is, $\operatorname{im}L \subset \widetilde H_\kappa^{n-r}(\mathcal F)$. Therefore it suffices to show
\begin{align}\label{4-17}
PH_\kappa^{n-r}(\mathcal F)\subset \widetilde H_\kappa^{n-r}(\mathcal F)\;.
\end{align}
Let $[\phi]\in PH_\kappa^{n-r}(\mathcal F)$; that is, $L^{r+1}[\phi]=0$.  Then there exists some $\psi \in \Omega_B^{n+r+1}(\mathcal F)$ such that 
\begin{align}\label{4-15}
L^{r+1}\phi = d_\kappa \psi\;.
\end{align}
Moreover,  by the representation theory of $\operatorname{sl}(2)$,  the map  $L^{r+1} :\Omega_B^{n-r-1}(\mathcal F) \to \Omega_B^{n+r+1}(\mathcal F)$ is an isomorphism. Hence   there exists $\eta\in \Omega_B^{n-r-1}(\mathcal F)$ such that 
\begin{align}\label{4-16}
\psi = L^{r+1}(\eta)\;.
\end{align}
From~\eqref{4-15} and~\eqref{4-16}, we get
\[
L^{r+1} (\phi- d_\kappa\eta)=0\;;
\]
that is, $\Lambda (\phi - d_\kappa\eta)=0$. Hence 
\[
\delta_\kappa(\phi-d_\kappa\eta) =[d_\kappa,\Lambda](\phi-d_\kappa\eta) =0\;.
\]
Therefore $\phi-d_\kappa\eta$ is a modified symplectic harmonic representative of $[\phi]$, showing~\eqref{4-17}.
\end{proof}

\section{Case of transverse K\"ahler foliation}
In this section, we consider a transverse K\"ahler foliation $(\mathcal F,J,\omega)$ of codimension $q=2n$ on a closed Riemannian manifol $M$.
Here $\omega$ is a basic  K\"ahler 2-form  and $J$ is a holonomy invariant almost complex structure 
on $Q$ such that  $\nabla J=0$,
where $\nabla$ is the transversal Levi-Civita connection on $Q$, extended in the usual way to tensors \cite{NT}. 

Let $\Lambda_{\mathbb C} Q^*$  be  the complexification of $\Lambda Q^*$. Then $\Lambda^1_{\mathbb C} Q^*= Q_{1,0} + Q_{0,1}$, where $Q_{1,0} $ and $Q_{0,1}$ are the eigenspaces of $J$ with eigenvalues $-i$ and $i$, respectively. 
Write $%
\kappa _{B}=\kappa _{B}^{1,0}+\kappa _{B}^{0,1}$, with 
\[
\kappa _{B}^{1,0}=\frac{1}{2}(\kappa _{B}+iJ\kappa _{B})\;,\quad\kappa _{B}^{0,1}=\overline{\kappa _{B}^{1,0}}\;,
\]%
where $\kappa_B$ is the basic part of the mean curvature form $\kappa$ \cite{Al}.
Note that there exists a bundle-like metric such that  $\partial_B^*\kappa^{1,0}=0$, where $d_B=\partial_B +\bar\partial_B$ on $\Omega_B^*(\mathcal F)\otimes \mathbb C$.  But we do not expect that $\partial_B\kappa_B^{0,1}$ would be in general zero for any metric \cite{JK}. Hence we get the following theorem.

\begin{thm}[Hard Lefschetz Theorem \cite{JK}]\label{HardLefschetzTheorem} 
Let $(\mathcal{F},J,\omega)$ be a transverse K\"{a}hler
foliation of codimension $2n$ on a closed Riemannian manifold  with compatible
bundle-like metric. Suppose that the class $\left[ \partial
_{B}\kappa _{B}^{0,1}\right] \in H_{\partial _{B}\bar{\partial}%
_{B}}^{1,1}(\mathcal{F}) $ is trivial. Then the hard Lefschetz
theorem holds for basic Dolbeault cohomology; that is, the map $L^{r}:H_{B}^{n-r}(\mathcal{F}) \rightarrow H_{B}^{n+r}(\mathcal{F})$ is an isomorphism. 
\end{thm}

On the other hand, we have the following.

\begin{thm}[See \cite{JK}]\label{t: 5-2}
Let $(\mathcal{F},J,\omega)$
be a transverse K\"{a}hler foliation of codimension $2n$ on a compact
Riemannian manifold with  compatible bundle-like metric. Then the following properties are
equivalent:
\begin{enumerate}[{\rm(1)}]

\item The class $\left[ \kappa _{B}\right] \in H_{B}^{1}(\mathcal{F})$ is trivial; that is, $(M,\mathcal{F})$ is taut.

\item The class $\left[ \partial _{B}\kappa _{B}^{0,1}\right] \in
H_{\partial _{B}\bar{\partial}_{B}}^{1,1}(\mathcal{F}) $ is
trivial.

\item The  hard Lefschetz theorem holds for basic Dolbeault cohomology.

\end{enumerate}
\end{thm}

From \Cref{t: 4-13,HardLefschetzTheorem,t: 5-2}, we get the following corollary.

\begin{cor} 
Let $(\mathcal{F},J,\omega)$ be a transverse K\"{a}hler
foliation of codimension $2n$ on a closed Riemannian manifold with compatible
bundle-like metric. Suppose that the class $\left[\partial_{B}\kappa _{B}^{0,1}\right] \in H_{\partial _{B}\bar{\partial}_{B}}^{1,1}(\mathcal{F})$ is trivial.  Then any basic cohomology class for $\mathcal F$ has a transversely symplectic harmonic representative.
\end{cor}

\begin{proof}  
Since $\left[ \partial_{B}\kappa _{B}^{0,1}\right]=0$ in $H_{\partial _{B}\bar{\partial}_{B}}^{1,1}(\mathcal{F}) $, by \Cref{t: 5-2}, we know that $[\kappa_B]=0$. So
\[
H_\kappa^*(\mathcal F) = H_B^*(\mathcal F)= H_T^*(\mathcal F)\;.
\]
Hence by \Cref{t: 4-13,HardLefschetzTheorem}, any basic cohomology class  has a transversely symplectic harmonic representative.
\end{proof}

\noindent{\em Acknowledgement.} The authors would like to thank  Professors  G.~Habib, Y.~Lin, J.I.~Royo Prieto and R.~Wolak   for  useful comments.


\begin{thebibliography}{99}

\bibitem{Al} J.A.~\'{A}lvarez L\'{o}pez, \emph{The basic component of the
mean curvature of Riemannian foliations}, Ann. Global Anal. Geom. 10
(1992), 179--194.

\bibitem{BC} L.~Bak and A.~Czarnecki, \emph{A remark on the Brylinski conjecture for orbifolds}, J. Aust. Math. Soc. 91 (2011), 1--12.


\bibitem{Bl}
D.E.~Blair, \emph{Riemannian geometry of contact and symplectic manifolds}, Birkh\"auser, 2010.

\bibitem{BR}
J.-L.~Brylinski, \emph{A differential complex for poisson manifolds}, J. differential Geom. 28 (1988), 93--114.
 

\bibitem{Cw}
L.A.~Cordero and R.A.~Wolak, \emph{Examples of foliations with foliated geometric structures}, Pacific J. math. 142 (1990), 265--276.

\bibitem{Ha}
K.~Habermann and L.~Habermann, \emph{Introduction to symplectic Dirac operators}, Lecture Notes in Mathematics 1887, Springer-Verlag Berlin, 2006. 

\bibitem{HR}
G.~Habib and K.~Richardson, \emph{Modified differentials and basic cohomology for Riemannian foliations}, J. Geom. Anal. 23 (2013), 1314-1342.

\bibitem{Ib} 
R.~Iba\~nez, \emph{Harmonic cohomology classes of almost cosymplectic manifolds}, Michigan Math. J. 44 (1997), 183--198.

\bibitem{EK} A.~El Kacimi-Alaoui, \emph{Op\'erateurs transversalement
elliptiques sur un feuilletage riemannien et applications}, Compositio Math. 
73  (1990), 57--106.


\bibitem{JK} S.D.~Jung and K.~Richardson, \emph{The mean curvature of transverse K\"{a}hler foliations},  Doc.  Math. 24 (2019), 995--1031.


\bibitem{KT}
F.~Kamber and Ph.~Tondeur, \emph{Foliations and harmonic forms, Harmonic mappings, twistors and $\sigma$-models}, Adv. Ser. Math. Phys. 4, World Sci. Publising, Singapore, 1988, 15--25.


\bibitem{KP} H.~Kitahara and H.K.~Pak, \emph{A note on harmonic forms on a compact manifold}, Kyungpook Math. J. 43 (2003), 1--10.

\bibitem{LI}
Y.~Lin, \emph{Hodge theory on transversely symplectic foliations}, Quart. J. Math. 69 (2018), 585--609.

\bibitem{LI1}
Y.~Lin, \emph{Lefschetz contact manifolds and odd dimensional symplectic geometry}, arXiv:1311.1431.

\bibitem{MA} 
O.~Mathieu, \emph{Harmonic cohomology classes of symplectic manifolds}, Comment. Math. Helvetici 70 (1995), 1--9.

\bibitem{NT}
S.~Nisikawa and Ph.~Tondeur, \emph{Transversal infinitesimal automorphisms for harmonic K\"ahler foliations}, Tohoku Math. J. 40 (1988), 599-611.

\bibitem{Pa} H.K.~Pak, \emph{Transversal harmonic theory for transversally symplectic flows}, J. Aust. Math. Soc. 84 (2008), 233--245.


\bibitem{RU} H.~Rummler,  \emph{Quelques notions simples en g\'eom\'etrie riemannienne et leurs applications aux feuilletages compacts}, Comment. Math. Helv.  54 (1979), 224-239.


\bibitem{TO} Ph.~Tondeur, \emph{Geometry of foliations}, Birkh\"{a}user, 1997.

\bibitem{Ya}
D.~Yan, \emph{Hodge structure on symplectic manifolds}, Adv. Math. 120 (1996), 143--154.

\end{thebibliography}
\end{document}